\documentclass[english]{ourlematema}
\usepackage{amsmath, amsthm, amssymb, hyperref, color}
\usepackage{graphicx}
\usepackage{caption}
\usepackage{verbatim}
\usepackage{chemfig, chemnum}
\usepackage{tikz}
\usepackage[linesnumbered,lined,commentsnumbered]{algorithm2e}
\usepackage{caption}
\usepackage{stmaryrd}
\usepackage{subcaption}
\usepackage{float}
\usepackage{mathtools}
\usepackage{braket}
\usepackage{anyfontsize}
\usepackage{todonotes}
\usepackage{tikz-cd}
\usepackage{cleveref}
\usepackage{csquotes}
\usepackage{enumitem}
\setenumerate{label=\textup{(\roman*)}}

\newtheorem{conjecture}{Conjecture}
\numberwithin{theorem}{section}

\newtheorem{examplex}[thm]{Example}

\renewenvironment{exa}
{\pushQED{\qed}\examplex}
{\popQED\endexamplex}
\definecolor{ourgreen}{rgb}{0, 0.7, 0}
\hypersetup{
    colorlinks=true,
    linkcolor=blue,
    filecolor= green,      
    urlcolor=blue,
    citecolor=ourgreen,
}

\newcommand{\ZZ}{\mathbb{Z}}

\newcommand{\RR}{\mathbb{R}}
\newcommand{\CC}{\mathbb{C}}
\newcommand{\PP}{\mathbb{P}}
\newcommand{\TT}{\mathbb{T}}
\newcommand{\T}{\textup{\textsf T}}

\newcommand{\transpose}{{\textsf{\textup{T}}}}

\DeclareMathOperator{\V}{V}
\DeclareMathOperator{\pr}{pr}
\DeclareMathOperator{\init}{in}
\renewcommand{\Im}{\operatorname{Im}}
\newcommand{\Ker}{\operatorname{Ker}}

\newcommand{\logdisc}{\nabla_{\log}}
\newcommand{\Hudisc}{\nabla_{\mathrm{Hu}}}

\newcommand{\reg}{\textup{reg}}

\DeclareMathOperator{\Ram}{Ram}
\DeclareMathOperator{\Branch}{Branch}
\DeclareMathOperator{\Fit}{Fit}
\DeclareMathOperator{\Crit}{Crit}

\DeclareMathOperator{\PGr}{\mathbb{G}}
\DeclareMathOperator{\GL}{GL}
\DeclareMathOperator{\diag}{diag}
\DeclareMathOperator{\Hu}{Hu}

\DeclareMathOperator{\im}{Im}

 \title{Logarithmic Discriminants of hyperplane arrangements}
 
  \author{Leonie Kayser}
  \address{%
  MPI for Mathematics in the Sciences, Leipzig \\
\email{leo.kayser@mis.mpg.de}
}

  \author{Andreas Kretschmer}
  \address{%
  Humboldt-Universität zu Berlin \\
\email{andreas.kretschmer@hu-berlin.de}
}

\author{Simon Telen}
\address{%
  MPI for Mathematics in the Sciences, Leipzig \\
\email{simon.telen@mis.mpg.de}
}

\begin{document}

\maketitle

\begin{abstract}
\noindent A recurring task in particle physics and statistics is to compute the complex critical points of a product of powers of affine-linear functions. The logarithmic discriminant characterizes exponents for which such a function has a degenerate critical point in the corresponding hyperplane arrangement complement. We study properties of this discriminant, exploiting its connection with the Hurwitz form of a reciprocal linear space.    
\end{abstract}


\section{Introduction}

We consider an arrangement ${\mathcal A}$ of $n+1$ hyperplanes $\ell_0(x) \cdots \ell_n(x) = 0$ in $\mathbb{C}^d$. Here, each $\ell_i$ is an affine-linear function in $x = (x_1, \ldots, x_d)$.
For any choice of $u = (u_0, \ldots, u_n) \in \CC^{n+1}$, one may locally define the following function of $x$: \[ {\mathcal L}_u(x) \,  \coloneqq \, \log \, ( \ell_0(x)^{u_0} \cdots \ell_n(x)^{u_n}) \, = \, u_0 \log \ell_0(x) + \cdots + u_n \log \ell_n(x). \]
The critical point equations of ${\mathcal L}_u$ are $d$ rational function equations in $x_1, \ldots, x_d$:
\begin{equation} \label{eq:criteqs}
 \frac{\partial {\mathcal L}_u}{\partial x_1} \, = \, \cdots \, = \, \frac{\partial {\mathcal L}_u}{\partial x_d} \, = \, 0.
\end{equation}
These are defined on the complement $X \coloneqq \mathbb{C}^d \setminus {\mathcal A}$ of our arrangement. We will assume that ${\mathcal A}$ is essential, meaning that a subset of the hyperplanes intersects in only one point. In that case, for generic values of $u$ there are $(-1)^d \cdot \chi(X)$ non-degenerate complex critical points \cite{orlik1995number, Huh2013}. Here, $\chi(\cdot)$ denotes the topological Euler characteristic. In other words, \eqref{eq:criteqs} has $(-1)^d \cdot \chi(X)$ isolated solutions, and the Hessian determinant of ${\mathcal L}_u$ does not vanish at these points. In fact, the same statement holds for any irreducible polynomials $\ell_0, \ldots, \ell_n$ so that $X = \mathbb{C}^d \setminus V(\ell_0 \cdots \ell_n)$ is very affine \cite{Huh2013}. In our setting, the number $(-1)^d \cdot \chi(X)$ can be expressed as the beta invariant of the matroid of ${\mathcal A}$, or, if the coefficients of each $\ell_i$ are real, as the number of bounded chambers in $X \cap \mathbb{R}^d$ \cite[Theorem~1.2.1]{varchenko1995critical}. 

This paper studies the discriminant of \eqref{eq:criteqs} in the parameters $u_0, \ldots, u_n$. More precisely, the logarithmic discriminant variety $\nabla_{\rm log}$ of ${\mathcal A}$ consists of all exponents $u \in \PP^n$ for which a complex critical point $x \in X$ of ${\mathcal L}_u$ is degenerate, meaning that the Hessian determinant of ${\mathcal L}_u$ vanishes at $x$. The formal definition of $\logdisc$ will be given in Section \ref{sec:definitions}.

In the context of positive geometry, the equations \eqref{eq:criteqs} appear as the scattering equations for bi-adjoint scalar $\phi^3$-theories in particle physics. In particular, the tree-level amplitude for such theories is a global residue over the complex critical points \cite[Theorem 13]{sturmfels2021likelihood}. This is called the CHY formula, after the authors of \cite{cachazo2014scattering}. For an $m$-particle scattering process, the arrangement ${\mathcal A}\subset \mathbb{C}^{m-3}$ is given by the non-constant $2 \times 2$-minors of the following $(2 \times m)$-matrix:
\begin{equation} \label{eq:matrixM0m}
    \begin{pmatrix}
    1 &  1 & 1 & 1 & \cdots & 1 & 0 \\ 
    0 &  1 & x_1 & x_2 & \cdots & x_{m-3} & 1
\end{pmatrix}
\end{equation} 
The complement $\mathbb{C}^{m-3} \setminus {\mathcal A}$ models the moduli space ${\mathcal M}_{0,m}$ of smooth projective genus zero curves with $m$ marked points \cite{sturmfels2021likelihood}. The function ${\mathcal L}_u$ is the scattering potential, and the $u_i$ are Mandelstam invariants. For a mathematical treatment, see \cite[Lecture 3]{lam2024moduli}.
\begin{exa} \label{ex:M05-intro}
    Five-particle scattering in bi-adjoint scalar $\phi^3$-theories leads us to consider the moduli space ${\mathcal M}_{0,5}$. It is isomorphic to the complement of five lines in $\mathbb{C}^2$, given by the affine-linear functions 
    \begin{equation} \label{eq:ellM05} \ell_0 \, = \, x_1, \quad \ell_1 \, = \, x_2, \quad \ell_2 \, = \, x_1-1, \quad \ell_3 \, = \, x_2 - 1, \quad \ell_4 \, = \, x_2 - x_1. \end{equation}
    The complement of the real arrangement ${\mathcal A} \cap \mathbb{R}^2$ consists of twelve components, two of which are bounded triangles. By Varchenko's theorem \cite[Theorem 1.2.1]{varchenko1995critical}, the equations \eqref{eq:criteqs} have two non-degenerate solutions for positive values of $u \in \mathbb{R}^5_+$, and there is one solution in each of the bounded triangles. The variety $\nabla_{\rm log}$ is defined by a homogeneous polynomial of degree four: 
    \begin{equation}  \label{eq:DeltaM05} \Delta_{\rm log} \, = \, (u_0u_3 + u_0u_4 + u_1u_4 + u_1u_2 + u_2u_4 + u_3u_4 + u_4^2)^2 - 4 u_0 u_1 u_2 u_3.\end{equation}
    This polynomial is called the logarithmic discriminant polynomial, or simply logarithmic discriminant. It exists whenever $\nabla_{\rm log}$ is a hypersurface. A symmetric determinantal expression for $\Delta_{\rm log}$ is given in Example \ref{ex:M05first}.
\end{exa}
In algebraic statistics, equations like \eqref{eq:criteqs} arise in maximum likelihood estimation for a discrete random variable with $n+1$ states. In our particular case, the functions $\ell_0, \ldots, \ell_n$ parametrize a $d$-dimensional linear statistical model in the $n$-dimensional probability simplex. Since the $\ell_i$ represent probabilities in that context, it is naturally assumed that $\ell_0(x) + \cdots + \ell_n(x) = 1$. Suppose that, in a statistical experiment, one observes state $i \in \{ 0, \ldots, n \}$ a total of $u_i$ times. The function ${\mathcal L}_u$ is the log-likelihood function, and its maximizer on the probability simplex is called the maximum likelihood estimate (MLE). This is the distribution in the model which makes the experimental observation $(u_0,\ldots,u_n)$ most likely. In particular, the MLE is among the complex critical points of ${\mathcal L}_u$. Motivated by this observation, the number of complex critical points of the log-likelihood function for generic data $u$ is called the maximum likelihood degree of the model. It measures the algebraic complexity of maximum likelihood estimation. Linear statistical models are discussed at the end of \cite[Section 1]{HuhSturmfels2014}. The logarithmic discriminant divides the data discriminant computed in \cite{rodriguez2015data}. The latter has additional factors which vanish at values of $u$ for which critical points move to the boundary of $X$ in a compactification. These were studied in \cite{sattelberger2023maximum}.

Our paper studies algebro-geometric invariants of the logarithmic discriminant, and it is accompanied by code for computing it \cite{mathrepo}. The following is a summary of our findings for logarithmic discriminants of generic arrangements. 

\begin{thm}[\ref{cor:positivity}, \ref{thm:logRamificationReducedIrreducible}, \ref{rem:logDiscEmpty}, \ref{prop:HurwitzDegree}, \ref{prop:AgreeAsSets}] \label{thm:main-intro}
    If $n = d$, then the logarithmic discriminant variety is empty. If $n \geq d+1$ and the coefficients of $\ell_0, \ldots, \ell_n$ are general, then the logarithmic discriminant $\logdisc(X) \subset \PP^n$ is irreducible and is cut out as a set by a homogeneous polynomial of degree $2d\binom{n-1}{d}$. If the $\ell_i$ are real, then that polynomial can be scaled to take positive values on~$\mathbb{R}^{n+1}_+$.
\end{thm}

The individual properties stated here for generic logarithmic discriminants hold under different genericity assumptions on the arrangement $\mathcal{A}$. A sufficient genericity condition for Theorem \ref{thm:main-intro} to hold is that the rank-$(d+1)$ matroid of ${\mathcal A}$ is uniform, and so is the rank-$d$ matroid obtained from the coefficients of $\ell_i$ standing with $x_1, \ldots, x_d$, dropping the constant terms. Irreducibility holds if $\mathcal{A}$ contains some subarrangement of $d+2$ hyperplanes having these properties.

\setcounter{thm}{0}
\begin{exa}[continued] \label{ex:M05-intro-2}
    The degree formula from Theorem~\ref{thm:main-intro} does not hold for special hyperplane arrangements such as the one in Example~\ref{ex:M05-intro}. Indeed, the logarithmic discriminant polynomial for five generic lines in $\mathbb{C}^2$ has degree twelve, while \eqref{eq:DeltaM05} has degree four. However, positivity on $\mathbb{R}^{n+1}_{+}$ still holds: The AM-GM inequality implies that $u_0u_3 + u_1u_2 \geq 2 \, \sqrt{u_0u_1u_2u_3}$ for non-negative $u \in \mathbb{R}^5_{\geq 0}$. This means that $\Delta_{\rm log}$ from \eqref{eq:DeltaM05} is positive on $\mathbb{R}^5_+$. However, $\Delta_{\rm log}$ is not globally non-negative on $\mathbb{R}^5$, for instance $\Delta_{\rm log}(-\frac{1}{2},1,2,-\frac{1}{2},-1) = -\frac{7}{16}$. 
\end{exa}

Our paper is outlined as follows. In Section~\ref{sec:definitions} we give a formal definition of the logarithmic discriminant for very affine hypersurface complements. Section~\ref{sec:hurwitz} introduces reciprocal linear spaces, their Hurwitz forms and explains the relation to the logarithmic discriminant. In Section~\ref{sec:positivity}, we prove positivity properties of $\Delta_{\rm log}$ on the positive orthant. In Section~\ref{sec:d=1} we give a complete description of the logarithmic discriminant of an arrangement of points on the line $\mathbb{C}^1$. In Section~\ref{sec:irreducibility} we discuss irreducibility of $\logdisc$. Section~\ref{sec:LogDiscHuDisc} ties back to Section~\ref{sec:hurwitz} and finishes the proof of Theorem~\ref{thm:main-intro}. Finally, Section~\ref{sec:M_0m} focuses on the case where $X = \mathcal{M}_{0,m}$, which is directly related to CHY theory for particle scattering.

\medskip

\textbf{Notation.}
Let $n\geq  d \geq 1$. We will always denote by $\mathcal{A} \subseteq \CC^d$ an essential non-central affine arrangement of $n+1$ affine hyperplanes defined by affine linear forms $\ell_0(x), \ldots , \ell_n(x)\in\CC[x_1,\dots,x_d]$. Except in Section~\ref{sec:definitions}, we will always denote by $X \coloneqq \CC^d \setminus \mathcal{A}$ its complement. We define matrices $L$ and $A$ as follows: 
\begin{equation*}
    (\ell_0(x), \ldots, \ell_n(x))^\T = Ax + b, \qquad L^\T \coloneqq [\,b \mid A \,] \in \CC^{(n+1)\times(d+1)}.
\end{equation*}
Essentiality and non-centrality of $\mathcal{A}$ together correspond to $L$ (and hence $A$) having full rank $d+1$ (and $d$). 
We say that a matrix is \emph{uniform} if so is its induced linear matroid, i.e., if all maximal minors are non-zero. We call $\mathcal{A}$ \emph{doubly uniform} if both $L$ and $A$ are uniform.
We use the notation $\llbracket n \rrbracket \coloneqq \{0,1,\ldots,n\}$. 
Our shorthand for the $n$-dimensional torus in $\PP^n$ will be $\TT$.

\vspace{-0.5cm}

\section{Definition of the logarithmic discriminant} \label{sec:definitions}

In this section we give a precise definition of the logarithmic discriminant of a very affine variety. We spell out the definition in our case of interest: essential hyperplane arrangement complements. 

Let $\TT^{n+1} = \Set{z \in \CC^{n+1} | z_0\dotsm z_n \neq 0 }$ be the algebraic torus of dimension $n+1$. A variety $X$ is \emph{very affine} if there is a closed embedding $\iota \colon X \hookrightarrow \TT^{n+1}$ for some $n$. Writing $\iota = (\iota_0,\dots,\iota_n)$, the log-likelihood function for fixed $u \in \CC^{n+1}$~is
\[
\mathcal{L}_u(x) = \log \iota_0(x)^{u_0} \dotsm \iota_n(x)^{u_n}, \qquad x \in X.
\]
The critical points $\Crit_X(u) \subseteq X_\reg$ are locally described by 
\[
\Crit_X(u) = \Set{x \in X_\reg | \frac{\partial\mathcal{L}_u}{\partial x_1}(x) = \dots = \frac{\partial\mathcal{L}_u}{\partial x_d}(x) = 0}.
\]

\begin{dfn}\label{def:naive disc}
The \emph{logarithmic discriminant} of $X \hookrightarrow \TT^{n+1}$ is the closure
\[
\logdisc(X) \coloneqq \overline{ \Set{u \in \PP^{n} | \Crit_X(u) \text{ is infinite or non-reduced}} }.
\]
\end{dfn}


If $\logdisc(X) \subset \mathbb{P}^n$ is a hypersurface, then its defining equation is unique up to scaling. We denote it by $\Delta_{\rm log}(X)$, or simply $\Delta_{\rm log}$. As we will see, $\logdisc(X)$ is not always a hypersurface. To describe its defining equations, we express $\logdisc(X)$ as the branch locus of a generically finite map of varieties. We formally introduce ramification and branch loci in the case of interest here.

Let $X,Y$ be $n$-dimensional smooth irreducible varieties and $f \colon X \to Y$ a morphism. Let $x \in X$, then $x$ is a reduced isolated point in $f^{-1}(f(x))$ if and only if the Jacobi matrix $J_f(x)$ has rank $n$. The \emph{ramification locus} $\Ram(f)$ is the subscheme locally defined by $\det J_f(x) = 0$. The \emph{branch locus} $\Branch(f) \subseteq Y$ is the closure of the scheme-theoretic image $f(\Ram(f))$. These loci are called singular scheme and discriminant scheme in \cite[Section~V.3]{eisenbud2006geometry}.

Now we can refine Definition~\ref{def:naive disc} for a model $X \hookrightarrow \TT^{n+1}$ as follows: Consider the \emph{likelihood correspondence} from algebraic statistics \cite[Definition 1.5]{HuhSturmfels2014}
\[
\mathcal{L}_X^\circ \coloneqq \Set{(u,x) | \nabla \mathcal{L}_u(x) = 0 } \subseteq \PP^n \times X_\reg.
\]
Let $f \colon \mathcal{L}_X^\circ \to \PP^n$ be the projection onto the first factor. This is a morphism of smooth irreducible varieties of dimension $n$ \cite[Theorem~1.6]{HuhSturmfels2014}.

\begin{prop} \label{prop:hessianeq}
The ramification locus of the projection $f\colon \mathcal{L}_X^\circ \to \PP^n$ is locally defined by the Hessian determinant of $\mathcal{L}_u(x)$:
\[
\det H_x(\mathcal{L}_u(x)) \coloneqq \det \bigg[\frac{\partial^2}{\partial x_j \partial x_k} \sum_{i=0}^n u_i \log \iota_i(x) \bigg]_{j,k=1}^d = 0.
\]
The logarithmic discriminant $\logdisc(X)$ is the branch locus of $f$.
\end{prop}

\begin{proof}
The ramification locus is defined by the determinant of the Jacobian of $\nabla \mathcal{L}_u(x)$, which is precisely the Hessian determinant. The description of $\logdisc(X)$ as the branch locus follows since $(u,x) \in \Ram(f)$ if and only if $x \in \Crit_X(u)$ is non-reduced or non-isolated, i.e., on a positive-dimensional component.
\end{proof}



We turn back to the setting of the Introduction, in which $X = \CC^d\setminus \mathcal{A}$ is the complement of $\mathcal{A} = \V(\ell_0\dotsm \ell_n) \subseteq \CC^d$, an arrangement of $n+1$ affine hyperplanes. If ${\mathcal A}$ is essential, then $X$ is very affine and a closed embedding into $\mathbb{T}^{n+1}$ is given by $\iota = (\ell_0,\dots,\ell_n)$. In this way Definition \ref{def:naive disc} agrees with the Introduction. We assume throughout that $\mathcal A$ is essential; if it is not, then $X$ is not very~affine.

The defining equations of the likelihood correspondence of $X = \CC^d \setminus \mathcal{A}$ with linear forms $(\ell_0(x),\dots,\ell_n(x))^\T = Ax + b$ can be written in a concise way as
\begin{equation}\label{eq:likecor}
\nabla \mathcal{L}_u(x) = A^\transpose \cdot \diag(1/\ell_0,\dots,1/\ell_n) \cdot u = 0.
\end{equation}

This also shows that $\mathcal{L}_X^\circ$ is a trivial projective bundle over $X$. The Hessian determinant can be computed using the Cauchy--Binet formula:
\begin{equation}\label{eq:cauchy-binet}
h = \det \bigg( A^\transpose \cdot \diag\Big(\frac{u_0}{\ell_0^2}, \dots, \frac{u_n}{\ell_n^2}\Big) \cdot A \bigg) = \sum_{I \subseteq \llbracket n \rrbracket,\ |I|=d } |A_I|^2 \frac{u^I}{(\ell^I)^2}.
\end{equation}
Here, $u^I = \prod_{i \in I} u_i$ and $|A_I|$ is the $d\times d$-minor of $A$ indexed by $I \subseteq \llbracket n \rrbracket = \{0,\dots,n\}$.

\vspace{-0.5cm}
\section{The Hurwitz discriminant} \label{sec:hurwitz}

Reformulating the equations \eqref{eq:likecor}, we see that for $u \in \PP^n$ and $x \in X = \CC^d \setminus \mathcal{A}$,
\[
x \in \Crit_X(u) \qquad \text{if and only if} \qquad  (u_0/\ell_0(x),\dots,u_n/\ell_n(x) )^\transpose  \in \Ker (A^\transpose).
\]
Let $\mathcal{R}_L \subseteq \PP^n$ be the image closure of the locally closed embedding
\begin{equation}\label{eq:gamma Linv embedding}
\gamma \colon \CC^d \setminus \mathcal{A} \to \PP^n, \qquad (x_1,\dots,x_d) \mapsto ({\ell_0(x)}^{-1}:\dots:{\ell_n(x)}^{-1} ).
\end{equation}
This is the \emph{reciprocal linear space} of the matrix $L = [\,b\mid A\,]^\T$. Note that $X = \CC^d \setminus \mathcal{A} \cong \im(\gamma)$. Define
\[
\varphi \colon \TT = \PP^n \setminus \V(u_0\dotsm u_n) \to \PGr(n-d,\PP^n), \quad u \mapsto  \Ker(A^\transpose \diag(u_0,\dots,u_n)).
\]
With this notation, the critical equations become linear equations on $\im(\gamma)$:
\begin{equation}\label{eq:crit iff tangent}
x \in \Crit_X(u) \qquad \text{if and only if} \qquad \gamma(x) \in \varphi(u) \cap  \im(\gamma) .
\end{equation}
Consider the incidence $\mathcal{I}^\circ = \Set{(\Lambda, y) | y \in \Lambda \cap \im(\gamma) } \subseteq \PGr(n-d,\PP^n) \times \mathcal{R}_L$
of complementary-dimensional linear spaces intersecting~$\mathcal{R}_L$, and denote by $\mathcal{L}_X^\circ |_\TT$ the restriction of $\mathcal{L}_X^\circ \subseteq \PP^n \times X$ to $\TT \times X$.

\begin{lemma} \label{lem:cartesiansec3}
Let $f\colon {\mathcal L}_X^\circ \rightarrow \mathbb{P}^{n}$ be as in Proposition \ref{prop:hessianeq}. The following commutative diagram is cartesian:
\[\begin{tikzcd}\label{cartesianDiagram}
	{\mathcal{L}_X^\circ|_\TT} & {\mathcal{I}^\circ} \\
	\TT & {\PGr(n-d,\PP^n)}
	\arrow["{\gamma \times \varphi}", from=1-1, to=1-2]
	\arrow["f"', from=1-1, to=2-1]
	\arrow["{\pr_{\Lambda}}"', from=1-2, to=2-2]
	\arrow[""{name=0, anchor=center, inner sep=0}, "\varphi", from=2-1, to=2-2]
	\arrow["\lrcorner"{anchor=center, pos=0.125}, draw=none, from=1-1, to=0]
\end{tikzcd}\]
If all  $n \times d$-submatrices $A^0,\dots,A^n$ of $A$ have full rank $d$, then $\varphi$ is defined on $D = \PP^n \setminus \bigcup_{i<j} \V(u_i,u_j)$ and the diagram is also cartesian with $D$ in place of $\TT$.
\end{lemma}

\begin{proof}
The first claim follows from Equation \eqref{eq:crit iff tangent}. Imposing the rank condition on $A$, we see that $A^\transpose \diag(u_0,\dots,u_n)$ still has rank $d$, so $\varphi$ is well-defined on $D$. The rest of the argument is the same.
\end{proof}

The branch locus of $\pr_{\Lambda}$ is the \emph{first associated hypersurface} $\mathcal{Z}_1(\mathcal{R}_L)$ \cite[Section 3.2]{gelfand2008discriminants}. 
It is an irreducible and reduced hypersurface defined by the \emph{Hurwitz form} $\Hu_{\mathcal{R}_L}$ in the Plücker ring $S(\PGr(n-d,\PP^n))$. The degree of $\Hu_{\mathcal{R}_L}$  is an invariant of the matroid $\mathcal{M}(\mathcal{A})$ of $\mathcal{A}$; for the uniform matroid it equals $2(n-d)\binom{n}{d-1}$ \cite{Sanyal2013, SturmfelsHurwitzForm}.

\begin{dfn} \label{def:hurdisc} The closure of the pullback of $\mathcal{Z}_1(\mathcal{R}_L)$ along $\varphi\colon \PP^n \dashrightarrow \PGr(n-d,\PP^n)$ is the \emph{Hurwitz discriminant} $\Hudisc(X) \coloneqq \varphi^{-1}(\mathcal{Z}_1(\mathcal{R}_L))$.
\end{dfn}

This is well-defined, since the complement of the maximal domain of definition of $\varphi$ has codimension at least $2$, so the hypersurface $\Hudisc(X) \subseteq \PP^n$ is uniquely determined. On $\TT$, one obtains the defining equation for $\nabla_{\rm Hu}$ from that of ${\mathcal Z}_1({\mathcal R}_L)$ by substituting $p_I = \det(A_{I}^\perp) \prod_{i \in I} u_i^{-1}$. Here, for $I \in \binom{\llbracket n \rrbracket}{n+1-d}$, the $p_I$ are the Pl\"ucker coordinates on $\mathbb{G}(n-d,\mathbb{P}^n)$ and $A^\perp$ represents $\ker(A^\T)$, i.e., $\im(A^\perp) = \ker(A^\transpose)$. The matrix $A^\perp_I$ consists of the rows of $A^\perp$ indexed by $I$.

Whenever $\Hudisc(X)$ is neither empty nor all of $\PP^n$, it is necessarily a hypersurface. The same need not be true for $\logdisc(X)$, as the following example shows.

\begin{exa}\label{example:reducibleAndCodim2}
    Let $n+1 = 6$ and $d = 3$. We consider the matrix 
    \begin{equation*}
        L \, = \,  [b \mid A]^\T \, = \, \begin{pmatrix}
            1 & 2 & 1 & 0 & 0 & 0 \\ 
            1 & 1 & 2 & 1 & 0 & 1 \\ 
            1 & \frac{3}{2} & \frac{3}{2} & 0 & 1 & 1\\
            0 & 0 & 0 & 1 & 1 & 2
        \end{pmatrix}.
    \end{equation*}
    A computation in \texttt{Macaulay2} \cite{M2} shows that
    \begin{align*}
        \logdisc(X) =& \V(144u_0^2+120u_0u_1+168u_0u_2+25u_1^2-70u_1u_2+49u_2^2) \\
        &\cup \V(u_3^2-2u_3u_4+4u_3u_5+u_4^2+4u_4u_5+4u_5^2) \\
        &\cup \V(u_0+u_1+u_2, u_3+u_4+u_5).
    \end{align*}
    The codimension~$1$ part of $\logdisc(X)$ is reducible and there is an additional linear codimension~$2$ component. In particular, $\logdisc(X)$ is not a hypersurface and therefore cannot agree with $\Hudisc(X)$ (not even as sets). For general $u$ in the codimension~$2$ component, ${\rm Crit}_X(u)$ has dimension~$1$. 
\end{exa}

Nonetheless, Equation~\eqref{eq:crit iff tangent} indicates that $\logdisc(X)$ and $\Hudisc(X)$ should be closely related, and indeed we always have the following.

\begin{prop}\label{prop:containment}
We have the inclusion $\logdisc(X) \cap \TT \subseteq \Hudisc(X) \cap \TT$ of schemes. If all submatrices $A^i$ have full rank $d$, then this also holds on $D$.
\end{prop}

For an instance in which the inclusion is strict and both $\nabla_{\rm Hu}$ and $\nabla_{\rm log}$ are hypersurfaces, see Example~\ref{ex:M05first}. Proposition~\ref{prop:containment} follows from applying the following general result to the diagram in Lemma \ref{lem:cartesiansec3}.

\begin{prop}
Given a cartesian diagram of smooth varieties
\[\begin{tikzcd}[column sep=tiny]
	{\Ram(f')} & {X'} & {\,} & X & {\Ram(f)} \\
	{\Branch(f')} & {Y'} && Y & {\Branch(f)}
	\arrow["\subseteq"{description}, draw=none, from=1-1, to=1-2]
	\arrow["{\varphi'}", from=1-2, to=1-4]
	\arrow["{f'}", from=1-2, to=2-2]
	\arrow["\lrcorner"{anchor=center, pos=0.125}, draw=none, from=1-2, to=2-4]
	\arrow["f", from=1-4, to=2-4]
	\arrow["\supseteq"{description}, draw=none, from=1-5, to=1-4]
	\arrow["\subseteq"{description}, draw=none, from=2-1, to=2-2]
	\arrow["\varphi", from=2-2, to=2-4]
	\arrow["\supseteq"{description}, draw=none, from=2-5, to=2-4]
\end{tikzcd}\]
we have that $\Ram(f') = \varphi'^{-1}(\Ram(f))$ and $\Branch(f') \subseteq \varphi^{-1}(\Branch(f))$ as subschemes of $X'$ and $Y'$ respectively.
\end{prop}
\begin{proof}
Let $\Omega_X$, $\Omega_Y$ be the cotangent bundles on $X$ and $Y$, respectively. Let $\delta\! f \colon f^*\Omega_Y \to \Omega_X$ be the dual differential map. Let $\Omega_f = \operatorname{Coker}(\delta\! f)$ be the sheaf of relative differentials. Then the ramification locus of $f$ as defined above coincides with the 0th Fitting ideal $\Fit_0(\Omega_f)$.

Using this perspective, the statement about ramification loci is the combination of  \cite[\href{https://stacks.math.columbia.edu/tag/01UX}{Lemma 01UX}]{stacks-project} and \cite[\href{https://stacks.math.columbia.edu/tag/0C3D}{Lemma 0C3D}]{stacks-project}. The statement about Branch loci follows as we have scheme-theoretic inclusions
\[
f'(\Ram(f')) = f'(\varphi'^{-1}(\Ram(f))) \subseteq \varphi^{-1}(f(\Ram(f))).   \qedhere
\]
\end{proof}


\section{Positivity of the logarithmic discriminant} \label{sec:positivity}

As a first application of the Hurwitz discriminant we show positivity of the logarithmic discriminant for real arrangements. For $i = 0, \dots, n$, let $H_i \coloneqq \overline{\V(\ell_i)} \subseteq \PP^n$ be the closures of the affine hyperplanes. The \emph{flats} of the matroid $\mathcal{M}({\mathcal A})$ of our arrangement are the various intersections of subsets of the $H_i$. We say that ${\mathcal A}$ \emph{has no flats at infinity} if no non-empty flats are contained in~$\mathbb{P}^d \setminus \mathbb{C}^d$.

\begin{thm}\label{thm:noFlatsAtInfinity}
If $\mathcal{A}$ has no flats at infinity and if $u \in (\CC^\times)^{n+1}$ is such that $\Crit_X(u)$ consists of $r = (-1)^d\cdot \chi(X)$ reduced points, then $u \notin \logdisc(X)$.
\end{thm}

\begin{proof}
It suffices to show that the set of such $u \in \CC^{n+1}$ is open. By the implicit function theorem, perturbing $u$ by a small $\varepsilon \in \CC^{n+1}$ has the effect that
\[
\Crit_X(u+\varepsilon) = \{\xi_1(\varepsilon),\dots,\xi_r(\varepsilon)\} \cup Z
\]
where $Z$ is a union of (potentially positive-dimensional) components not containing the $\xi_i(\varepsilon)$, and the latter are pairwise distinct. We want to show that $Z = \emptyset$.
The locally closed embedding $\gamma$ from \eqref{eq:gamma Linv embedding} realizes $\Crit_X(u+\varepsilon)$ as the intersection $\im(\gamma) \cap \varphi(u+\varepsilon)$. If ${\mathcal A}$ has no flats at infinity, then ${\mathcal R}_L \subseteq \PP^n$ is a projective variety of degree $r$ \cite[Corollary 3.9]{betti2024proudfoot}. Hence, by a version of the Bézout inequality \cite[Theorem 8.28]{Brgisser1997}, the intersection $\im(\gamma) \cap \varphi(u+\varepsilon)$ has at most $r$ components. Hence $Z=\emptyset$ and $\Crit_X(u+\varepsilon)$ consists of $r$ reduced points.
\end{proof}

\begin{cor}\label{cor:positivity}
Let $\mathcal{A} \subseteq \CC^d$ be a real arrangement with $r$ bounded real regions. Then for any $u \in \RR_+^{n+1}$, $\Crit_X(u)$ consists of $r$ reduced points. If $\mathcal{A}$ has no flats at infinity, then $\logdisc \cap \RR_+^{n+1} = \emptyset$.
\end{cor}
\begin{proof}
The first statement is a theorem of Varchenko \cite[Theorem 1.2.1]{varchenko1995critical}. The second claim follows from the first and Theorem~\ref{thm:noFlatsAtInfinity}.
\end{proof}

\section{Arrangements of points on the line} \label{sec:d=1}

In this section we study the situation for $d=1$, that is, arrangements of $n+1\geq 3$ distinct points $\mathcal{A}$ on the affine line $\CC^1$. In this case, we are able to provide a complete description. Similar statements in higher dimensions may not hold or are harder to prove. Below, the map $f \colon \mathcal{L}_X^\circ \rightarrow \mathbb{P}^n$ is as in Proposition \ref{prop:hessianeq}.

\begin{thm}\label{thm:d=1 case}
The discriminant $\logdisc(L) \subseteq \PP^n$ of an arrangement of $n+1 \geq 3$ points in $\mathbb{C}^1$ is an irreducible reduced hypersurface of degree $2(n-1)$. The class of $\Ram(f)$ in $\operatorname{A}^\bullet(\PP^n \times \PP^1) = \ZZ[\alpha,\beta]/\langle \alpha^{n+1},\beta^2\rangle$ equals $\alpha^2 + 2(n-1)\alpha\beta$.
\end{thm}

\begin{exa}\label{exa:3pts in C1}
Consider three linear forms $(x,\,x+1,\,x+b)$ with $b \notin \{0,1\}$, then 
\[
\Delta_{\log}(L) = (b-1)^2u_{0}^{2} + 2b(b-1)u_{0}u_{1} + b^{2}u_{1}^{2} - 2(b-1)u_{0}u_{2} + 2bu_{1}u_{2} + u_{2}^{2}.
\]
This ternary quadric in $u_0,u_1,u_2$ has discriminant $-4b^2(b-1)^2$, hence is smooth and irreducible for any choice of $b \notin \{ 0,1 \}$.
\end{exa}

Without loss of generality we can assume $\ell_i(x) = x + b_i$, since the discriminant is independent of scaling the linear forms. The ramification locus is defined in $\PP^n \times (\CC^1 \setminus \{-b_0,\dots,-b_n\})$ by the two equations
\begin{align*}
0 &= g_1(u;x) = \frac{u_0}{x + b_0} + \dots + \frac{u_n}{x + b_n}, \\
0 &= g_2(u;x) =  \frac{u_0}{(x + b_0)^2} + \dots + \frac{u_n}{(x + b_n)^2}. 
\end{align*}

\begin{lemma} \label{lem:RamBranchbirat}
The ramification locus $\Ram(f)$ is a smooth $(n-1)$-dimensional irreducible variety. The projection $f \colon \Ram(f) \to \Branch(f)$ is birational. Hence, $\logdisc \subseteq \PP^n$ is an irreducible and reduced hypersurface.
\end{lemma}

\begin{proof}
The cone over $\Ram(f)$ in $\CC^{n+1} \times X = \CC^{n+1} \times (\CC^1 \setminus \{-b_0,\dots,-b_n\})$ is defined by $\langle g_1,g_2 \rangle \subseteq R[u_0,\dots,u_n;x]_{\ell_0\dotsm\ell_n}$. Substituting $u_0$ and $u_1$, we see
\[
\frac{\CC[u_0,\dots,u_n;x]_{\ell_0\dotsm\ell_n}}{\langle g_1,g_2 \rangle} \, \cong \,  \frac{\CC[u_1,\dots,u_n;x]_{\ell_0\dotsm\ell_n}}{\langle \tilde{g}_2 \rangle} \,  \cong \, \CC[u_2,\dots,u_n;x]_{\ell_0\dotsm\ell_n},
\]
where $\tilde{g}_2 = g_2$ after substituting $u_0$. This shows smoothness and irreducibility of $\Ram(f)$. Since $\dim X = 1$, the projection is generically finite. It remains to show generic injectivity. For this, consider the following variety:
\[
\mathcal{B} \coloneqq \Set{(u,x,y) \in \PP^n \times X\times X | \!\!\!\begin{array}{c}
x\neq y, \\ g_1(u;x)=g_2(u;x)=g_1(u;y)=g_2(u;y)=0
\end{array}\!\!\! }.
\]
A direct symbolic computation shows that this is a $\PP^{n-4}$-bundle over the two-dimensional irreducible variety $\set{ (x,y) \in X \times X | x \neq y }$ for $n \geq 4$ (for smaller $n$ $\mathcal{B}$ is empty). The \texttt{Macaulay2} code to verify this can be found at \cite{mathrepo}. It follows that $\mathcal{B}$ has dimension at most $n-2$, and its projection to $\nabla_{\rm log}$ cannot be dominant.
\end{proof}

\begin{proof}[Proof of Theorem \ref{thm:d=1 case}]
We first study the intersection of $g_1$ and $g_2$ in $\PP^n \times \PP^1$. Clearing denominators yields the bihomogeneous equations
\[
g_1^{\rm h}(u;x_0,x_1) = \sum_{i=0}^n u_i \prod_{k \neq i} (x_1+b_kx_0), \qquad g_2^{\rm h}(u;x_0,x_1) = \sum_{i=0}^n u_i \prod_{k \neq i} (x_1+b_kx_0)^2.
\]
Let $W_i = \V(g_i^{\rm h}) \subseteq \PP^n \times \PP^1$. 
In the boundary $\PP^n\times \{-b_0,\dots,-b_n,\infty\} \subseteq \PP^n\times\PP^1$ we have
\[
\langle g_1^{\rm h},g_2^{\rm h}\rangle = \begin{cases}
\langle u_i,x_1+b_ix_0\rangle & \text{in }\mathcal{O}_{\PP^n\times\PP^1,\V(u_i)\times \{-b_i\}}, \\
\langle u_0+\dots+u_n, x_0\rangle & \text{in }\mathcal{O}_{\PP^n\times\PP^1,\V(u_0+\dots+u_n)\times \{\infty\}}.
\end{cases}
\]
This gives the decomposition into irreducible reduced components
\[
W_1 \cap W_2 = \Ram(f) \cup  \bigg( \bigcup_{i=0}^n \V(u_i) \times \{-b_i\} \bigg) \cup \big(\V(u_0+\dots+u_n) \times \{\infty\} \big).
\]
Each of the $n+2$ terms on the right has class $\alpha \beta$. We also have
\[
[W_1] = \alpha + n\beta, \quad [W_2] = \alpha + 2n\beta, \quad [W_1 \cap W_2] = \alpha^2 + 3n\alpha \beta.
\]
Combining the above displays we find the cycle class of ${\rm Ram}(f)$: 
\[
[\Ram(f)] = \alpha^2 + (3n - (n+2))\alpha\beta = \alpha^2 + 2(n-1)\alpha \beta.
\]
Pushing forward via $\PP^n\times \PP^1 \to \PP^n$, we obtain
\[
\deg \logdisc = \frac{\text{coeff.\ of $\alpha\beta$}}{\deg(f|_{\Ram(f)}) } \overset{\ref{lem:RamBranchbirat}}{=} 2(n-1). \qedhere
\]
\end{proof}

\begin{cor}\label{cor:flat family resultant}
The logarithmic discriminant defines a flat family of hypersurfaces parametrized by $\mathcal{M}_{0,n+2}$
\[
\Set{(u,\mathcal{A}) \in \PP^n\times \mathcal{M}_{0,n+2} | u \in \logdisc(\CC^1 \setminus \mathcal{A})} \to \mathcal{M}_{0,n+2}.
\]
We have formulae $\Delta_{\log}(X) = \operatorname{Disc}_x(g_1^{\rm h}) = \frac{1}{u_0\dotsm u_n(u_0+\dots+u_n)} \operatorname{Res}_x(g_1^{\rm h},g_2^{\rm h})$.
\end{cor}
\begin{proof}
The fibers of this family are hypersurfaces of constant degree $2(n-1)$ by Theorem~\ref{thm:d=1 case}, hence it is a flat family.
It follows from Definition \ref{def:naive disc} that $\logdisc(X)$ is contained in $\V(\operatorname{Disc}_x(g_1^{\rm h}))$ and $\V(\operatorname{Res}_x(g_1^{\rm h},g_2^{\rm h}))$. On the other hand the discriminant of a degree $n$ polynomial has degree $2(n-1)$ in its coefficients, so the two must agree. Similarly the resultant of $g_1^{\rm h}$ and $g_2^{\rm h}$ has degree $3n$ and splits off the given $n+2$ factors.
\end{proof}

The discriminant in Example \ref{exa:3pts in C1} is smooth for any point configuration. This fails for $n +1 > 3$, as a computation of the discriminant of $\Delta_{\log}(X)$ shows:

\begin{cor}
For any arrangement of four distinct points in $\CC^1$ the logarithmic discriminant is a singular quartic surface in $\PP^3$.
\end{cor}


\section{Irreducibility} \label{sec:irreducibility}

In this section we prove a sufficient criterion for the logarithmic discriminant to be an irreducible variety.


\begin{thm}\label{thm:logRamificationReducedIrreducible} 
Let $\mathcal{A} \subseteq \CC^d$ be an arrangement of $n+1 \geq d+2$ hyperplanes given by a matrix $L = [\,b\mid A\,]^\T$. Assume that $L$ has a $(d+1) \times (d+2)$ submatrix which is uniform and for which also the corresponding $d \times (d+2)$ submatrix of $A^\T$ is uniform. Then $\Ram(\mathcal{L}_X^\circ \rightarrow \PP^n)$ and $\logdisc(X)$ are reduced and irreducible.

\end{thm}

The sufficient condition of Theorem~\ref{thm:logRamificationReducedIrreducible} is not necessary for irreducibility. An example is the arrangement $\mathcal{M}_{0,5}$ of $n+1=5$ lines in $\CC^2$ (Example~\ref{ex:M05-intro}). We have also seen that in degenerate cases $\logdisc(X)$ may be reducible, even with several components of codimension 1 (Example~\ref{example:reducibleAndCodim2}).

\begin{rem}\label{rem:logDiscEmpty}
If $n+1 = d+1$, then $\logdisc(X)$ is empty. Indeed, up to an affine transformation on $\CC^d$, there is only a single essential non-central arrangement, namely $\V(x_1\dotsm x_d(x_1+\dots+x_d+1))$. For this arrangement the critical locus $\Crit_X(u)$ is easily seen to be a single reduced point, so $\logdisc(X)$ is empty.
\end{rem}

\begin{lemma}\label{lemma:linearlyIndependentProducts}
Let $S \coloneqq \CC[x_0,\dots,x_d]$ and consider linear forms $\{\tilde{\ell}_0,\dots,\tilde{\ell}_{d+1}\} \subseteq S_1$. The set of pairwise products $P \coloneqq \set{\tilde{\ell}_i\tilde{\ell}_j | i<j} \subseteq S_2$ is linearly independent if all subsets of $d+1$ forms of $\{\tilde{\ell}_0,\dots,\tilde{\ell}_{d+1}\}$ are linearly independent.
\end{lemma}

\begin{proof}
Up the action of $\GL_{d+1}(\CC)$ on $S_1$ we may assume $\tilde{\ell}_0 = x_0, \dots, \tilde{\ell}_d = x_d$. Then $\tilde{\ell}_{d+1} = c_0x_0+\dots+c_dx_d$ with all $c_i \neq 0$. It is then straightforward to see that $P$ spans $S_2$. Since $|P| = \binom{d+2}{2} = \dim_\CC S_2$, $P$ is linearly independent.
\end{proof}



\begin{proof}[Proof of Theorem~\ref{thm:logRamificationReducedIrreducible}]
Without loss of generality, we assume that the submatrix of the first $(d+2)$ columns of $L$ satisfies the hypotheses. For convenience we may also assume that the leftmost $d \times d$ submatrix of $A^\T$ is the identity matrix.

We only need to prove that the ramification scheme $\Ram(f)$ is reduced and irreducible. Its scheme-theoretic defining equations in $\PP^n\times X$ are given by the polynomials in \eqref{eq:likecor}, \eqref{eq:cauchy-binet}. By making the reversible substitution $v_i = u_i/\ell_i(x)$, this scheme is isomorphic to the one defined by
\begin{equation*}
    A^\T v = 0, \qquad \det\big(A^\T \diag(v_0/\ell_0,\dots,v_n/\ell_n) \, A\big) = \sum_{\substack{I\subseteq \llbracket n \rrbracket \\ |I|=d}} |A_I|^2 \frac{v^I}{\ell^I} = 0.
\end{equation*}
By assumption the linear equations $A^\T v = 0$ have the form $v_j = \sum_{i=d}^n -a_{ij} v_i$ for $j=0,\dots,d-1$. Linearly substituting $v_0,\dots,v_{d-1}$ and clearing denominators by multiplying with $\ell_0\dotsm\ell_n$ yields the polynomial $\Tilde{h}$ given by
\begin{equation}\label{eq:htilde}
\Tilde{h}  = \sum_{\substack{I\subseteq \llbracket n \rrbracket \\ |I|=d}} |A_I|^2 \, \tilde{v}^I \, \ell^{\llbracket n \rrbracket \setminus I}, \quad \tilde{v} = (-\sum_{i=d}^n a_{i,0}v_i,\dots,-\sum_{i=d}^n a_{i,(d-1)}v_i,v_d,\dots,v_n).
\end{equation}

The discussion so far shows that $\V(\Tilde{h}) \subseteq \PP^{n-d}\times X$ is isomorphic to $\Ram(f)$.
Since $\CC[v_d,\dots,v_n,x_1,\dots,x_d]_{\ell_0\dotsm \ell_n}$ is a UFD, our goal is thus to show that $\Tilde{h}$ is irreducible. For $n+1=d+2$ we are going to show the (slightly stronger) claim of irreducibility in $\CC[v_d,v_{d+1},x_1,\dots,x_d]$.

We first reduce to the case $n+1=d+2$ as follows: $\tilde{h}$ is homogeneous in $v$, and hence if $\tilde{h} = p\cdot q$ factors, then these factors are necessarily homogeneous too. Thus if there is a factorization into non-units, then setting $v_{d+2},\dots,v_n$ to zero still yields a non-trivial decomposition of $\tilde{h}' \coloneqq \tilde{h}(v_d,v_{d+1},0,\dots,0;x)$. Inspecting \eqref{eq:htilde}, we see that $\Tilde{h}'$ is the $\Tilde{h}$ of the sub-arrangement of the first $d+2$ multiplied by the unit $\ell_{d+2}\dotsm \ell_n$. Hence, it suffices to consider $n+1=d+2$.

Let $\Tilde{h}^{\rm h}(v_d,v_{d+1};x_0,\dots,x_n)$ be the homogenization of $\Tilde{h}$ with respect to the new variable $x_0$. We also have the \enquote{naive} homogenization $N(v_d,v_{d+1};x_0,\dots,x_n)$ obtained from \eqref{eq:htilde} in which each $\ell_i$ is replaced by its homogenization $\ell_i^{\rm h} = b_ix_0 + a_{i,1}x_1+\dots+a_{i,d}x_d$. $N$ has $x$-degree two and must be of the form $N = \tilde{h}^{\rm h} \cdot x_0^e$ for some $e\geq 0$. On the other hand, substituting $x_0=0$ in $N$ yields identically zero:
\begin{align} \label{eq:hessianvanishes}
\begin{split}
    &A^\T \diag(v) \diag(\ell_0^{\rm h}(0,x),\ldots,\ell_{d+1}^{\rm h}(0,x))^{-1} A x \\
    & = A^\T \diag(v) (1,1,\ldots,1)^\T \, = \, A^\T v \,  = \,  0.
    \end{split}
\end{align}
This shows $e\geq 1$ and hence $\tilde{h}^{\rm h}$ has $x$-degree $\leq 1$.

By the previous discussion $\tilde{h}^{\rm h}$ is bihomogeneous of bidegree $(d,1)$ or $(d,0)$. In both cases any non-trivial factorization must involve $q \in \CC[v_{d+1},v_{d+2}]$, which has a non-zero root $(\lambda_0,\lambda_1) \in \CC^2\setminus 0$. We show that this leads to a contradiction.

The polynomial $s(x) \coloneqq \tilde{h}^{\rm h}(\lambda_0,\lambda_1;x_0,\dots,x_d)$ vanishes identically. Since the matrix $L$ is uniform, by Lemma \ref{lemma:linearlyIndependentProducts} all pairwise products $(\ell^{\rm h})^J$, $|J|=2$, are linearly independent. From this we conclude that all coefficients of $s(x)$ in front of $(\ell^{\rm h})^J$ must vanish. From \eqref{eq:htilde} we see that these coefficients are 
\[
\operatorname{coeff}[(\ell^{\rm h})^{J}](s)= |A_{\llbracket d+1 \rrbracket \setminus J}|^2 \cdot w^{\llbracket d+1 \rrbracket \setminus J}, \qquad w \coloneqq \tilde{v}|_{(v_d,v_{d+1})=(\lambda_0,\lambda_1)} \in \CC^{d+2}.
\]
Since $A^\T$ is uniform we conclude that products of $d-1$ entries of $w$ vanish, so $w$ has at least three zero entries. But then $w = 0$, because $Aw = 0$ and no $d-1$ columns of $A^\T$ are dependent. This is a contradiction to $(w_d,w_{d+1}) = \lambda \neq 0$. 
\end{proof}

\begin{rem}\label{rmk:HurwitzRamification}
    Under the same assumptions as in Theorem~\ref{thm:logRamificationReducedIrreducible}, the Hurwitz ramification scheme $\Ram(\pr_{\Lambda} \colon \mathcal{I}^\circ \rightarrow \PGr(n-d,\PP^n))$ is also reduced and irreducible. Irreducibility follows from standard projective geometry techniques and only uses that $X$ is smooth. Reducedness, on the other hand, is a consequence of Theorem~\ref{thm:logRamificationReducedIrreducible} because the ramification scheme of $f \colon \mathcal{L}_X^\circ \rightarrow D$ is the scheme-theoretic preimage of the Hurwitz ramification scheme, and both are locally defined by a single equation. If the equation of the latter was a pure power with exponent $>1$, then so would be the equation of the former.
\end{rem}

\section{Logarithmic discriminants from Hurwitz forms} \label{sec:LogDiscHuDisc}

This section is devoted to the doubly uniform case. We compute the degree of $\Hudisc(X)$ and prove that $\logdisc(X)$ and $\Hudisc(X)$ agree as sets. By the \emph{Newton polytope} of a hypersurface we mean that of its defining equation, see Section \ref{sec:M_0m}.

\begin{prop}\label{prop:HurwitzDegree}
    Let both $L$ and $A$ be uniform. Then
    \begin{equation*}
        \deg \Hudisc(X) = 2 d \binom{n-1}{d},
    \end{equation*}
    and the Newton polytope of $\Hudisc(X)$ is the full dilated standard simplex.
\end{prop}

\begin{proof}
    By \cite{ProudfootSpeyerBrokenCircuit}, the circuits of the matrix $L$ determine natural generators of the ideal $I$ of $\mathcal{R}_L$ which even form a universal Gröbner basis. Since all maximal minors of $L$ are non-zero, the generators of $I$ are homogeneous square-free polynomials of degree $d+1$ with precisely $d+2$ terms each. More precisely, every subset $T \subseteq \llbracket n \rrbracket = \{0, \ldots, n\}$ of cardinality $|T| = d+2$ defines a generator 
    \begin{equation} \label{eq:PSgenerators}
        g_T \, = \, \sum_{i \in T} \lambda_i \prod_{j \in T \setminus i} y_j \in \CC[y_0,\ldots,y_n].
    \end{equation}
    Here, $\lambda_i$ is the coefficient of the $i$-th column vector of $L$ in a linear relation for the $d+2$ columns of $L$ indexed by $T$; this linear relation is unique up to scaling and all coefficients $\lambda_i$ are non-zero.
    
    Let $w \in \ZZ_{>0}^{n+1}$ be a weight vector with distinct integer entries. The $w$-initial ideal $\init_w(I)$ is the square-free monomial ideal generated by the $w$-leading monomials of the generators \eqref{eq:PSgenerators}. Explicitly, we have ${\rm in}_w(g_T) = \prod_{j \in T \setminus i_T} y_j$, where $w_{i_T}$ is minimal among $\set{ w_j | j \in T }$.  Let $w_i$ be the smallest entry of $w$. It follows from \cite[Corollary~4.4]{SturmfelsHurwitzForm} that with respect to $w$, the initial monomial of $\Hu_{\mathcal{R}_L}$ is
    \begin{equation*}
        \prod_{\substack{S \subseteq \llbracket n \rrbracket \\ i \in S , \, |S| = d}} p_{\llbracket n \rrbracket \setminus S}^{2(n-d)},
    \end{equation*}
    where $p_{\llbracket n \rrbracket \setminus S}$ denotes the Plücker variable corresponding to the complement of~$S$. Now, since $A$ is uniform, this monomial does not vanish at $\Ker(A^\T \diag(u))$ if only $u_i$ is zero. Therefore, the preimage of $\Hu_{\mathcal{R}_L}$ via $\varphi$ does not contain $\V(u_i)$ for any $i$. To compute the preimage on $\TT$, we proceed as in the discussion following Definition \ref{def:hurdisc}. We substitute $p_{\llbracket n \rrbracket \setminus S} = \det(A_{\llbracket n \rrbracket \setminus S}^\perp) \prod_{i \in \llbracket n \rrbracket \setminus S} u_i^{-1}$. 
    By the assumption on $A$, all determinants $\det(A_{\llbracket n \rrbracket \setminus S}^\perp)$ are nonzero. Hence, we obtain a non-zero scalar multiple of
    \begin{equation*}
        \left( u_1 \cdots \hat{u_i} \cdots u_n \right)^{- 2(n-d) \binom{n-1}{n-d}} =  \left( u_1 \cdots \hat{u_i} \cdots u_n \right)^{- 2d \binom{n-1}{d}}.
    \end{equation*}
    Therefore, multiplying by $(u_1 \cdots u_n)^{2d \binom{n-1}{d}}$, we obtain $u_i^{2d \binom{n-1}{d}}$, as desired.
\end{proof}

\begin{prop}\label{prop:AgreeAsSets}
    Let both $L$ and $A$ be uniform. Then $\logdisc(X)$ and $\Hudisc(X)$ agree set-theoretically. In particular, $\logdisc(X)$ is a hypersurface.
\end{prop}

To prove Proposition \ref{prop:AgreeAsSets}, we need some more notation and a lemma. Let $\mathcal{I} \subseteq \PGr(n-d,\PP^n) \times \mathcal{R}_L$ be the closure of $\mathcal{I}^\circ$ from Lemma \ref{lem:cartesiansec3}. In fact, $\mathcal{I} = \set{(\Lambda,y) | y \in \Lambda } \subseteq \PGr(n-d,\PP^n) \times \mathcal{R}_L$. Furthermore, let $Y$ be the incidence
\begin{equation*}
    Y = \Set{(u,y) |  A^\transpose \diag(y) u = 0 } \subseteq D \times \mathcal{R}_L,
\end{equation*}
where $D = \PP^n \setminus \bigcup_{i<j} \V(u_i,u_j)$. These fit into the following diagram:
\begin{equation}\label{diagramProper}
\begin{tikzcd}
	Y & {\mathcal I} \\
	{D \times \mathcal R_L} & {\PGr(n-d,\PP^n)\times \mathcal R_L} \\
	D & {\PGr(n-d,\PP^n)}
	\arrow[from=1-1, to=1-2]
	\arrow[hook, from=1-1, to=2-1]
	\arrow["\lrcorner"{anchor=center, pos=0.125}, draw=none, from=1-1, to=2-2]
	\arrow[hook, from=1-2, to=2-2]
	\arrow["{\varphi \times \operatorname{id}}", from=2-1, to=2-2]
	\arrow["{\pr_u}", from=2-1, to=3-1]
	\arrow["\lrcorner"{anchor=center, pos=0.1}, draw=none, from=2-1, to=3-2]
	\arrow["{\pr_\Lambda}"', from=2-2, to=3-2]
	\arrow["\varphi", from=3-1, to=3-2]
\end{tikzcd}
\end{equation}
Each square in this diagram is cartesian.

\begin{lemma} \label{lem:torusinvpoints}
    Let $L$ be uniform and let $\varphi\colon D \rightarrow \mathbb{G}(n-d,\mathbb{P}^n)$ be as in Lemma \ref{lem:cartesiansec3}. 
    The  image of $\mathcal{I} \cap \left( \Im(\varphi) \times (\mathcal{R}_L \setminus \TT) \right)$ under the projection $\mathbb{G}(n-d,\PP^n) \times \mathcal{R}_L \rightarrow \mathcal{R}_L$ is the union of the $n+1$ torus-fixed points $e_i$ of $\PP^n$.
\end{lemma}

\begin{proof}
    Since $L$ is uniform, it follows from the explicit description of the generators of the ideal of $\mathcal{R}_L$ (see \cite{ProudfootSpeyerBrokenCircuit}) that the set-difference $\mathcal{R}_L \setminus \TT$ is the union of all coordinate linear subspaces of $\PP^n$ of codimension $n+1-d$. If $A^\T \diag(u) y = 0$ and at least $n+1-d$ entries of $\diag(u) y$ are zero, then necessarily $\diag(u) y = 0$, so $y$ has at most one non-zero entry since $u \in D$ has at most one zero entry.
\end{proof}

Notice that $Y$ is smooth of dimension $n$ on the complement of the torus fixed points $e_i$ in $\mathcal{R}_L$ because over $\mathcal{R}_L \cap \TT$ it is an open subset of a projective bundle via $\pr_y\colon Y \rightarrow \mathcal{R}_L$. Moreover, $\pr_y^{-1}(e_i) = \V(u_i) \times \{e_i\}$ which has dimension $n-1$. Therefore, these fibers cannot be irreducible components since $Y$ is cut out by $d$ equations in $D \times \mathcal{R}_L$. We conclude that $Y$ is irreducible and, since $\mathcal{R}_L$ is Cohen--Macaulay \cite{ProudfootSpeyerBrokenCircuit} and $D$ is smooth, $Y$ is Cohen--Macaulay and thus reduced.

\begin{proof}[Proof of Proposition \ref{prop:AgreeAsSets}]

Let $\mathfrak{R} \subseteq \mathcal{I}$ be the scheme-theoretic closure of the Hurwitz ramification scheme $\Ram(\pr_\Lambda\colon \mathcal{I}^\circ \rightarrow \PGr(n-d,\PP^n))$. It is reduced and irreducible of codimension $1$ in $\mathcal{I}$ by Remark~\ref{rmk:HurwitzRamification}.
Let $Z$ be the preimage of $\mathfrak{R}$ in $Y$. Notice that $Z$ contains the closure of the ramification locus $\Ram(f\colon \mathcal{L}_X^\circ \rightarrow \PP^n)$. Its projection $\pr_u(Z)$ is $\nabla_{\rm Hu}$, which is a hypersurface by Proposition \ref{prop:HurwitzDegree}. Hence, if $Z$ has a component of codimension $>1$, then its image in $D$ under $\pr_u$ is contained in the image of some codimension $1$ component of $Z$.
Hence, it suffices to prove that $Z$ does not contain any codimension $1$ component of the complement $Y \setminus \mathcal{L}_X^\circ$. The latter is the union of $n+2$ codimension $1$ components: The $n+1$ fibers $\pr_y^{-1}(e_i)$ as well as the component $Y \cap (D \times \mathcal{R}_{A^\T})$. Here, $\mathcal{R}_{A^\T}$ is the $(d-1)$-dimensional reciprocal linear space corresponding to the matrix $A^\T$.

Now, the image of $\pr_y^{-1}(e_i)$ in $D$ is $\V(u_i)$ by the proof of Lemma \ref{lem:torusinvpoints}. On the other hand, by Proposition~\ref{prop:HurwitzDegree}, $\V(u_i)$ is not a component of $\Hudisc$.

For $Y \cap (D \times \mathcal{R}_{A^\T})$, we first restrict to the open $Y' \coloneqq Y \cap (D \times (\mathcal{R}_L \cap \TT))$ which is smooth. Likewise, we can restrict $\mathcal{I}$ to the smooth open $$\mathcal{I}' \coloneqq \mathcal{I} \cap \left( \PGr(n-d,\PP^n) \times (\mathcal{R}_L \cap \TT) \right).$$ Then $\mathfrak{R} \cap \mathcal{I}'$ is the ramification scheme of $\mathcal{I}' \rightarrow \PGr(n-d,\PP^n)$, and its preimage $Z \cap Y'$ is the ramification scheme of $\pr_u\colon Y' \rightarrow D$.

Hence, we only need to prove that $Y' \rightarrow D$ is not ramified everywhere along $Y' \cap ( D \times \mathcal{R}_{A^\T})$. The latter is an irreducible subvariety of $Y'$ of codimension~$1$ because it is an open of a projective bundle via $\pr_y$.
We now proceed in a similar fashion as in the proof of Theorem~\ref{thm:logRamificationReducedIrreducible}: First identify $\mathcal{R}_L \cap \TT \cong \PP^d \setminus \overline{\mathcal{A}}$ where $\overline{\mathcal{A}}$ is the closure of $\mathcal{A}$ in $\PP^d$. Under this isomorphism, $\mathcal{R}_{A^\T} \cap \TT$ corresponds to the hyperplane at infinity $\{x_0 = 0\}$. In the affine open chart where, for instance, $x_1 = 1$, the ramification scheme of $Y' \rightarrow D$ is defined by the determinant of the Jacobian matrix of $A^\T \diag(u) (\ell_0(x)^{-1}, \ldots, \ell_n(x)^{-1})^\T = 0$ but now with respect to the variables $x_0, x_2, \ldots, x_d$. Let $A'$ be the $(n+1) \times d$ submatrix of $L^\T = [b \mid A]$ obtained by deleting the column corresponding to $x_1$. After the substitution $v_i = u_i/\ell_i(x)$, the ramification locus is defined by
\begin{equation*}
    A^\T v = 0, \qquad \det(A^\T \diag(v) \diag(\ell_0(x), \ldots, \ell_n(x))^{-1} A') = 0
\end{equation*}
in the chosen chart. 
We may assume that the first $d \times d$ submatrix of $A$ is the identity matrix. 
The substituted Hessian $\Tilde{h}$ is defined as in the proof of Theorem~\ref{thm:logRamificationReducedIrreducible}. It only uses the variables $v_d, v_{d+1}, \ldots, v_n$. 
As in the proof of Theorem~\ref{thm:logRamificationReducedIrreducible}, we will write $N(v_d, v_{d+1}, \ldots, v_n; x_0, x_1, \ldots, x_d)$ for the ``naive'' homogenization of $\tilde{h}$.
The coefficient of $v_{d}^d$ in $N$ is
\begin{equation*}
    (-1)^{d-1} \cdot \Big( \prod_{j = 0}^{d-1}a_{d,j}  \Big) \cdot \Big( \prod_{k= d+1}^n \ell_k^{\rm h}(x) \Big) \cdot  \Big( \sum_{i = 0}^{d} \det(A_{\llbracket d \rrbracket \setminus i}) \det(A'_{\llbracket d \rrbracket \setminus i}) \frac{\ell^{\rm h}_i(x)}{a_{d,i}} \Big) ,
\end{equation*}
where $a_{d,i}$ denotes the entry of $A$ for $0 \leq i \leq d-1$ and we set $a_{d,d} \coloneqq -1$. The sum on the right is the only factor that can possibly be zero for $x \in \PP^d \setminus \overline{\mathcal{A}}$. It is enough to prove that this sum does not vanish identically after setting $x_0 = 0$. 

By Equation \eqref{eq:hessianvanishes}, the sum vanishes identically at $x_0 = 0$ if $A'$ is replaced by $A$. 
Note that $\ell^{\rm h}_0(0,x_1,\ldots,x_d), \ldots, \ell^{\rm h}_{d+1}(0,x_1,\ldots,x_d)$ form a \emph{circuit} of the matrix $A$. Hence, the coefficients of a linear relation among them are unique up to scaling. Therefore, if also the above sum vanishes identically after setting $x_0 = 0$, then there is some $\lambda \in \CC$ such that $\det(A'_{\llbracket d \rrbracket \setminus i}) = \lambda  \det(A_{ \llbracket d \rrbracket \setminus i})$ for all $i = 0, \ldots, d$. But this implies that the inverse of the first $(d+1) \times (d+1)$ submatrix of $L^\T$ has two linearly dependent columns, which is impossible.
\end{proof}

\vspace{-1cm}

\section{The logarithmic discriminant of \texorpdfstring{${\mathcal M}_{0,m}$}{M\_0,m}} \label{sec:M_0m}

This section focuses on the case where the complement of the arrangement ${\mathcal A}$ models the moduli space ${\mathcal M}_{0,m}$ of genus zero curves with $m$ distinct marked points. A point in ${\mathcal M}_{0,m}$ is represented by a $2 \times m$ matrix of the form \eqref{eq:matrixM0m}
whose $2 \times 2$-minors are nonzero. The columns represent homogeneous coordinates for the marked points in $\mathbb{P}^1$. Modulo automorphisms of $\mathbb{P}^1$, we may assume that the first points are $0, 1$, and the last point is $\infty$. The remaining points are represented by the coordinates $x_i$, and nonzero minors implies that our $m$ points are distinct. This represents ${\mathcal M}_{0,m}$ as $\mathbb{C}^{m-3} \setminus {\mathcal A}$, where ${\mathcal A}$ is the arrangement of hyperplanes defined by the non-constant minors of our $2 \times m$-matrix. The parameters $d,n$ are $d = m-3$ and $n =\frac{m(m-3)}{2}$. Example~\ref{ex:M05-intro} illustrates this for $m = 5$. 

These arrangements are relevant for scattering amplitudes in bi-adjoint scalar $\phi^3$-theories, as studied by Cachazo--He--Yuan \cite{cachazo2014scattering}. It is customary to denote the exponents $u$ by $s_{ij}$ in this context, in such a way that $s_{ij}$ is the exponent standing with the determinant of columns $i$ and $j$ in \eqref{eq:matrixM0m}. In physics, each column of \eqref{eq:matrixM0m} corresponds to a particle with momentum $p_i$ and $s_{ij}$ are \emph{Mandelstam} \emph{invariants}. 
For more background, we refer to \cite[Section 3.2]{lam2024moduli}. 

As we have already observed for $m = 5$ in Example~\ref{ex:M05-intro-2}, the arrangement ${\mathcal A}$ associated to ${\mathcal M}_{0,m}$ is not a generic arrangement of $n+1$ hyperplanes in $\mathbb{C}^d$. This reflects, for instance, in the low degree of its logarithmic discriminant. The scattering equations \eqref{eq:criteqs} also have a small number of solutions compared to generic arrangements: the signed Euler characteristic $(-1)^d \cdot \chi(X)$ of the complement $X = \mathbb{C}^d \setminus {\mathcal A}$ of a generic arrangement of $n+1$ hyperplanes in $\mathbb{C}^d$ equals $\binom{n}{d}$. In contrast, it is well-known that $(-1)^{m-3} \cdot \chi({\mathcal M}_{0,m}) = ( m-3)!$ \cite[Proposition 1]{sturmfels2021likelihood}. For $m = 10$, these numbers are $5379616$ and $5040$, respectively.

The following conjecture holds for $m = 5$ by Example~\ref{ex:M05first} below. It is supported by a numerical computation for $m =6, 7,8$. The code is available at \cite{mathrepo}. 

\begin{conjecture} \label{conj:1}
    For any $m \geq 5$, the logarithmic discriminant $\nabla_{\rm log}({\mathcal M}_{0,m})$ is an irreducible and reduced hypersurface in $ \mathbb{P}^{\frac{m(m-3)}{2}-1}$. For $m = 5, 6, 7,8$ its degree equals $4, 30, 208$ and $1540$, respectively. 
\end{conjecture}


\begin{exa} \label{ex:M05first} Using Mandelstam variables, the polynomial $\Delta_{\log}$ in \eqref{eq:DeltaM05} reads 
\begin{equation*}  \Delta_{\rm log} \, = \, (s_{13}s_{24} + s_{13}s_{34} + s_{14}s_{34} + s_{14}s_{23} + s_{23}s_{34} + s_{24}s_{34} + s_{34}^2)^2 - 4 s_{13}s_{14}s_{23}s_{24}.\end{equation*}
Explicitly, the substitution is $u = (u_0, u_1, u_2, u_3,u_4) = (s_{13},s_{14},s_{23},s_{24},s_{34})$. This polynomial captures when the following equations have a singular~solution: 
\begin{equation} \label{eq:scatteringM05}
    \frac{s_{13}}{x_1}  + \frac{s_{23}}{x_1-1} - \frac{s_{34}}{x_2-x_1} \, = \, \frac{s_{14}}{x_2}  + \frac{s_{24}}{x_2-1} + \frac{s_{34}}{x_2-x_1}  \, = \, 0.
\end{equation}
It was pointed out to us by Sebastian Mizera that the discriminant $\Delta_{\rm log}$ has a nice symmetric determinantal expression: 
\[ \Delta_{\rm log} \, = \, \det \begin{pmatrix}
    0 & -s_+ & s_{13} & s_{14}\\ - s_+ & 0 & s_{23} & s_{24} \\ 
    s_{13} & s_{23} & 0 & s_{34} \\ s_{14} & s_{24} & s_{34} & 0
\end{pmatrix},\]
where $s_+ = s_{13} + s_{14} + s_{23} + s_{24} + s_{34}$. In physics, this is a principal minor of the $5 \times 5$ \emph{Gram matrix} $G = (s_{ij})_{1 \leq i < j \leq 5}$ associated to our particles. The entries satisfy the identities $s_{ii} = 0, s_{ij} = s_{ji}$ and $\sum_{j=1}^5s_{ij} = 0, i = 1, \ldots, 5$. Using these linear relations, all $s_{ij}$ are expressed in terms of the entries of our five exponents, e.g., $s_{12} = - s_+$. This way all five principal $4 \times 4$-minors of $G$ evaluate to $\Delta_{\rm log}$.

The discriminant is obtained algebraically as follows. 
First, one isolates $x_2$ in the first equation of \eqref{eq:scatteringM05} and substitutes the resulting expression $x_2 = f(x_1)$ in the second equation. This gives a rational function with numerator
\[ ( s_{13}(x_1-1) + s_{23}x_1 ) (c(u) + b(u) \, x_1 \, + a(u) \, x_1^2) \, = \, 0,  \]
where $a, b$ and $c$ are homogeneous polynomials of degree 2 in $u$. 
 The factor $( s_{13}(x_1-1) + s_{23}x_1 )$ is the denominator in the expression $f(x_1)$ found for $x_2$, so it cannot vanish at a critical point of ${\mathcal L}_u$. 
The quadratic discriminant $b(u)^2 - 4\, a(u)c(u)$ equals $\Delta_{\log}$. For $u \in \nabla_{\rm log}$, the degenerate critical point is 
\[ x_1 \, = \, \frac{-b(u)}{2\, a(u)}, \quad x_2 \, = \, \frac{-b(u) \, (2\, a(u) \, (s_{13}+s_{34})+b(u) \, (s_{13}+s_{23}+s_{34}))}{2\, a(u) \, (2\, a(u)s_{13} + b(u)(s_{13}+s_{23}))}. \]
The Hurwitz form of the reciprocal linear space ${\mathcal R}_L$, where $L \in \mathbb{Z}^{3 \times 5}$ contains the coefficients of the $\ell_i$ in \eqref{eq:ellM05}, is a hypersurface in $\mathbb{G}(2,\mathbb{P}^4)$. It is represented modulo the Pl\"ucker ideal by a homogeneous polynomial of degree eight in $p_{012}, p_{013}, \ldots, p_{234}$. This polynomial has 2285 terms, and can be downloaded at \cite{mathrepo}. The Hurwitz discriminant from Section \ref{sec:hurwitz} is obtained by substituting $p_{ijk} = \det(A^\perp_{ijk}) \cdot (u_iu_ju_k)^{-1}$, see the discussion below Definition \ref{def:hurdisc}. Performing this substitution we obtain the reducible polynomial
\[ \Delta_{\rm Hu} \, = \, (s_{13} + s_{23} + s_{34})^2 
\cdot (s_{14} + s_{24} + s_{34})^2 \cdot \Delta_{\rm log}.  \qedhere \]
\end{exa}
Such an explicit analysis is out of reach for $m = 6$: $\Delta_{\rm log}({\mathcal M}_{0,6})$ is (conjectured to be) a polynomial of degree $30$ in nine variables. We have not been able to compute all its coefficients. However, we computed partial information which showcases an interesting recursive structure. 

Our strategy for obtaining such \emph{partial information} is based on \emph{degenerations} of the logarithmic discriminant variety $\nabla_{\rm log}$, which we assume to have codimension one. Concretely, for a weight vector $w \in \mathbb{R}^{n+1}_{\geq 0}$, let 
\[ \Delta_{\rm log}^{w}(u) \, = \, \lim_{\varepsilon \rightarrow 0} \varepsilon^{- \min_{\alpha \in P} w \cdot \alpha } \cdot \Delta_{\rm log}(\varepsilon^w \cdot u), \quad \text{where } \varepsilon^w \cdot u \, = \, (\varepsilon^{w_0}u_0, \ldots, \varepsilon^{w_n}u_n).\]
This is called the \emph{$w$-initial form} of $\Delta_{\rm log}$. Geometrically, for each nonzero value of $\varepsilon$ we consider the hypersurface $\nabla_{\rm log}(\varepsilon) = \set{u \in \mathbb{P}^{n} | \Delta_{\rm log}(\varepsilon^w \cdot u) = 0 }$. As $\varepsilon \rightarrow 0$, this converges to $\nabla_{\rm log}^w = \set{ u\in\mathbb{P}^n | \Delta_{\rm log}^w(u) = 0 }$. 
Next, we interpret $\Delta_{\rm log}^w$ in terms of the \emph{Newton polytope} of $\Delta_{\rm log}$. This will justify the hope that $\Delta_{\rm log}^w$ is easier to compute. 

Recall that the Newton polytope of $\Delta_{\rm log}$ is the convex polytope $P \subset \mathbb{R}^{n+1}$ obtained as the convex hull of all exponents appearing in the polynomial $\Delta_{\rm log}$. By homogeneity, it has dimension at most $n$. A \emph{face} of $P$ is a subset of the form $P^w = \set{ \alpha' \in P | w \cdot \alpha' = \min_{\alpha \in P} w \cdot \alpha }$ for some $w \in \mathbb{R}_{\geq 0}^{n+1}$. Facets are faces of dimension $\dim P - 1$ and  vertices are zero-dimensional faces.


\begin{exa}
    The Newton polytope $P$ of $\Delta_{\rm log}$ for ${\mathcal M}_{0,5}$ is four-dimensional with $f$-vector $(7,17,18,8)$. Its eight facets are $P^w$ for $w$ in the list 
\[ \begin{matrix} 
(1,0,1,0,1),& (0,1,0,1,1),&(0,0,1,1,1),&(1,1,0,0,1),\\
(1,0,0,0,0),&(0,1,0,0,0),&(0,0,1,0,0),&(0,0,0,1,0).
\end{matrix} \qedhere\]
\end{exa}

The relation between initial forms of $\Delta_{\rm log}$ and faces of $P$ is as follows: $\Delta_{\rm log}^w(u)$ is the sum of all terms in $\Delta_{\rm log}$ whose exponents lie on the face $P^w$. In particular, $\Delta_{\rm log}^w$ is a nonzero polynomial whose degree is that of $\Delta_{\rm log}$.  

The challenge is to compute $\Delta_{\rm log}^w$ without computing $\Delta_{\rm log}$ first. This can be done by degenerating the ramification locus, as illustrated in the next examples.

\begin{exa} \label{ex:softlimitM05}
For ${\mathcal M}_{0,5}$, the initial form of the facet normal $w=(0,1,0,1,1)$ is
\[  \Delta_{\rm log}^{(0,1,0,1,1)} \, = \, (s_{13}s_{24} + s_{13}s_{34} + s_{14}s_{34} +  s_{23}s_{34})^2 - 4 s_{13}s_{14}s_{23}s_{24}.\]
We will obtain this from the likelihood equations $\nabla {\mathcal L}_{\varepsilon^{w} \cdot u} = 0$: 
\begin{equation*} 
    \frac{s_{13}}{x_1}  + \frac{s_{23}}{x_1-1} - \frac{\varepsilon s_{34}}{x_2-x_1} \, = \, \frac{\varepsilon s_{14}}{x_2}  + \frac{\varepsilon s_{24}}{x_2-1} + \frac{\varepsilon s_{34}}{x_2-x_1}  \, = \, 0.
\end{equation*}
We study these equations in the limit for $\varepsilon \rightarrow 0$. This choice of $w$ leads to the \emph{soft limit} for the particle labeled $4$ in the physics literature \cite{cachazo2020singular}. This is because we replace $s_{ij}$ by $\varepsilon \,s_{ij}$ each time $4 \in ij$. As $\varepsilon$ approaches $0$, the two complex critical points converge in ${\mathcal M}_{0,5}$ \cite[Section 2]{cachazo2020singular}. Their limits are the two solutions~of
\begin{equation} \label{eq:scatteringM05degen}
    \frac{s_{13}}{x_1}  + \frac{s_{23}}{x_1-1}  \, = \, \frac{s_{14}}{x_2}  + \frac{s_{24}}{x_2-1} + \frac{s_{34}}{x_2-x_1}  \, = \, 0.
\end{equation}
The condition for the solutions of \eqref{eq:scatteringM05degen} to coincide  is precisely $\Delta_{\rm log}^w = 0$. To check this, one eliminates $x_1$ and computes a quadratic discriminant as in Example \ref{ex:M05first}. Thus, $\nabla_{\rm log}^{w}$ is the projection of the ramification locus in the soft limit.
\end{exa}

\begin{exa}
    In Example \ref{ex:softlimitM05}, the critical points converge in ${\mathcal M}_{0,5}$ in the soft limit. The situation for the facet $w=(0,0,1,1,1)$ is different. 
    We consider
\begin{equation*} 
    \frac{s_{13}}{x_1}  + \frac{\varepsilon s_{23}}{x_1-1} - \frac{\varepsilon s_{34}}{x_2-x_1} \, = \, \frac{s_{14}}{x_2}  + \frac{\varepsilon s_{24}}{x_2-1} + \frac{\varepsilon s_{34}}{x_2-x_1}  \, = \, 0.
\end{equation*}
This time, when $\varepsilon \rightarrow 0$, the critical points move to the boundary of ${\mathcal M}_{0,5}$. With the Ansatz $x_1 = 1 + \varepsilon \bar{x}_1 + O(\varepsilon^2)$ and $x_2 = 1 + \varepsilon \bar{x}_2 + O(\varepsilon^2)$, we find 
\begin{equation*} 
    s_{13} + \frac{s_{23}}{\bar{x}_1} - \frac{s_{34}}{\bar{x}_2-\bar{x}_1} \, = \, s_{14}  + \frac{s_{24}}{\bar{x}_2} + \frac{ s_{34}}{\bar{x}_2-\bar{x}_1}  \, = \, 0.
\end{equation*}
We compute that the solutions for $(\bar{x}_1,\bar{x}_2)$ coincide if and only if $\Delta_{\rm log}^w = 0$. This happens when the critical points of ${\mathcal L}_{\varepsilon^{w} \cdot u}$ approach the point $(1,1)$ in the same tangent direction. Thus, the initial form $\Delta_{\rm log}^w$ detects when these critical points collide for $\varepsilon \rightarrow 0$ in the Deligne--Mumford compactification $\overline{{\mathcal M}}_{0,5}$. 
\end{exa}

We suggest a generalization of the observation in Example \ref{ex:softlimitM05} concerning the initial forms of $\Delta_{\rm log}({\mathcal M}_{0,m})$ corresponding to soft limits. For $3 \leq k \leq m-1$, let $w_k \in\mathbb{R}^{n+1}_+$ be the weight vector corresponding to the soft limit for particle $k$. That is, $w$ assigns weight $1$ to $s_{ij}$ if $k \in ij$, and weight $0$ otherwise. Example \ref{ex:softlimitM05} uses $k=4, m =5$. Let ${\mathcal M}_{0,m-1}^{(k)}$ be the complement of the hyperplane arrangement in $\mathbb{C}^{m-4}$ given by all non-constant minors of the submatrix of \eqref{eq:matrixM0m} obtained by deleting the $k$-th column. If its logarithmic discriminant is a hypersurface, then $\Delta_{\rm log}({\mathcal M}_{0,m-1}^{(k)})$ is a polynomial in the Mandelstam variables $s_{ij}$ with $k \notin ij$. 

    Based on numerical computations and nearly complete theoretical insights, we believe that $\Delta_{\rm log}({\mathcal M}_{0,m-1}^{(k)})^{m-3}$ divides the soft limit initial form $\Delta_{\rm log}^{w_k}({\mathcal M}_{0,m})$. We leave a proof of this claim for future work, and end with an example.
    


\begin{exa} \label{ex:softlimitM06}
    The scattering equations for ${\mathcal M}_{0,6}$ with $u \rightarrow \varepsilon^{w_5}\cdot u$ are
    \begin{align*}
        f_1  \, = \, \frac{s_{13}}{x_1} + \frac{s_{23}}{x_1-1} - \frac{s_{34}}{x_2-x_1} - \frac{\varepsilon \,s_{35}}{x_3-x_1} &\, =\,0\\
        f_2  \, = \, \frac{s_{14}}{x_2} + \frac{s_{24}}{x_2-1} + \frac{s_{34}}{x_2-x_1} - \frac{\varepsilon \,s_{45}}{x_3-x_2} &\, =\, 0,\\
        f_3  \, =\,\frac{\varepsilon \, s_{15}}{x_3} + \frac{\varepsilon \, s_{25}}{x_3-1} +\frac{\varepsilon \, s_{35}}{x_3-x_1} + \frac{\varepsilon \, s_{45}}{x_3-x_2} & \, =\,0.
    \end{align*}
    The discriminant of $f_1=f_2=0$ with $\varepsilon = 0$ is the logarithmic discriminant $\Delta_{\rm log}({\mathcal M}_{0,5}^{(5)})$ computed in Example \ref{ex:M05first}. It appears as a factor $\Delta_{\rm log}({\mathcal M}_{0,5}^{(5)})^3$ in $\Delta^{w_5}_{\rm log}({\mathcal M}_{0,6})$. 
    The degree of $\Delta_{\rm log}({\mathcal M}_{0,6})$ is (conjecturally) 30. The missing factor of degree 18 captures when $(x_1,x_2)$ leads to a double root in $x_3$ for $f_3=0$. To compute this, first isolate $x_2$ in $f_1 =0$ to express $x_2$ in terms of $x_1$ in $f_2$ and $f_3$. After that, up to some spurious factors, our degree 18 factor is the resultant ${\rm Res}_{x_1}(g_2,{\rm Disc}_{x_3}(g_3))$, with $g_i$ the numerator of $f_i$. The face $P^{w_5}$ of the Newton polytope $P$ of $\Delta_{\rm log}({\mathcal M}_{0,6})$ is a facet. It 
    contains $600413$ lattice points and has $f$-vector $(237,907,1432,1195,564,149,20)$ .
\end{exa}
\begin{footnotesize}
\noindent {\bf Funding statement}:
This project started at a workshop held at MPI MiS Leipzig, supported by the European Union (ERC, UNIVERSE PLUS, 101118787).
Views and opinions expressed are however those of the authors only and do not necessarily reflect those of the European Union or the European Research Council Executive Agency. Neither the European Union nor the granting authority can be held responsible for them.
\end{footnotesize}

\vspace{-1.3cm}

\bibliography{references}

\vspace{-0.5cm}
\end{document}